\documentclass{article}
\usepackage{geometry}
\usepackage{graphicx}	
\usepackage[cp1251]{inputenc}
\usepackage[english]{babel}
\usepackage{amsfonts,amssymb,mathrsfs,amscd,amsmath,amsthm}
\usepackage{verbatim}
\DeclareGraphicsExtensions{.pdf,.png,.jpg}

\def\thtext#1{
  \catcode`@=11
  \gdef\@thmcountersep{. #1}
  \catcode`@=12
}

\def\threst{
  \catcode`@=11
  \gdef\@thmcountersep{.}
  \catcode`@=12
}
\def\opt{{\operatorname{opt}}}

\theoremstyle{plain}
\newtheorem{thm}{Theorem}
\newtheorem{prop}{Proposition}[section]
\newtheorem{cor}[prop]{Corollary}
\newtheorem{ass}[prop]{Assertion}
\newtheorem{lem}[prop]{Lemma}

\theoremstyle{definition}
\newtheorem*{examp}{Example}
\newtheorem{dfn}[prop]{Definition}
\newtheorem{rk}[prop]{Remark}

 \catcode`@=11
 \def\.{.\spacefactor\@m}
 \catcode`@=12

\newcommand{\dl}{\delta}
\newcommand{\D}{\Delta}
\newcommand{\e}{\varepsilon}
\newcommand{\g}{\gamma}

\renewcommand{\l}{\lambda}

\renewcommand{\r}{\rho}
\newcommand{\s}{\sigma}

\renewcommand{\t}{\tau}

\def\opt{{\operatorname{opt}}}

\newcommand{\cP}{\mathcal{P}}

\newcommand{\N}{\mathbb{N}}

\newcommand{\R}{\mathbb{R}}

\def\:{\colon}

\newcommand{\rom}[1]{{\em #1}}

\renewcommand{\)}{\rom)}

\renewcommand{\:}{\colon}

\renewcommand{\c}{\circ}

\newcommand{\oPi}{\stackrel{\raise-2pt\hbox{$\c$}}\Pi}
\newcommand{\oW}{\stackrel{\raise-2pt\hbox{$\c$}}W}

\renewcommand{\ss}{\subset}
\newcommand{\x}{\times}

\newcommand{\diam}{{\operatorname{diam}}\,}

\renewcommand{\min}{{\operatorname{min}}\,}

\newcommand{\dis}{{\operatorname{dis}}\,}

\newcommand{\id}{\operatorname{id}}
\newcommand{\cR}{\mathcal{R}}
\newcommand{\cM}{\mathcal{M}}

\def\GH{\operatorname{\mathcal{G\!H}}}

\def\cB{{\cal B}}
\def\cC{{\cal C}}

\def\cH{{\cal H}}

\begin{document}
 \title{Path Connectivity of Spheres in the Gromov--Hausdorff Class}
\author{A.~Ivanov, R.~Tsvetnikov, and A.~Tuzhilin}
\maketitle

\begin{abstract}
The paper is devoted to geometrical investigation of Gromov--Hausdorff distance on the classes of all metric spaces and of all bounded metric spaces. The main attention is paid to path connectivity questions. The path connected components of the Gromov--Hausdorff class of all metric spaces are described, and the path connectivity of spheres is proved in several particular cases.

\end{abstract}

\section*{Introduction}
\markright{Introduction}
Comparison of metric spaces is an important problem that is interesting as in pure research, so as in numerous applications. A natural approach is to define some distance function between metric spaces: The less similar the spaces are, the more the distance between them should be. Currently, the Gromov--Hausdorff distance and some its modifications are the most commonly used for these purposes.

The history of this distance goes back to works of F.~Hausdorff~\cite{Hausdorff}. In~1914 he defined a non-negative symmetrical function on pairs of subsets of a metric space $X$ that is equal to the infimum of the  reals  $r$ such that the first subset is contained in the $r$-neighborhood of the second one, and vice-versa. It turns out that this function satisfies the triangle inequality and further, it is a metric on the family of all closed bounded subsets of $X$. Later on, D.~Edwards~\cite{Edwards} and, independently, M.~Gromov~\cite{Gromov} generalized the Hausdorff's construction to the case of all metric spaces using their isometric embeddings into all possible ambient metric spaces, see formal definition below. The resulting function is now referred as the \emph{Gromov--Hausdorff distance}. Notice that this distance is also symmetric, non-negative and satisfies the triangle inequality. It always equals zero for a pair of isometric metric spaces, therefore in this context such spaces are usually identified. But generally speaking, the Gromov--Hausdorff distance can be infinite and also can vanish on a pair of non-isometric spaces. Nevertheless, if one restricts himself to the family $\cM$ of isometry classes of all compact metric spaces, then the Gromov--Hausdorff distance satisfies all metric axioms. The set $\cM$ endowed with the Gromov--Hausdorff distance is called the \emph{Gromov--Hausdorff space}. Geometry of this metric space turns out to be rather non-trivial and is studied actively. It is well-known that $\cM$ is path connected, Polish (i.e., complete and separable), geodesic~\cite{INT}, and also the space $\cM$ is not proper, and has no non-trivial symmetries~\cite{ITSymm}. The Gromov--Hausdorff distance is effectively used in computer graphics and computational geometry for comparison and transformation of shapes, see for example~\cite{MemSap}. It also has applications in other natural sciences, for example, C.~Sormani used this distance to prove stability of Friedmann cosmology model~\cite{Sor}. A detailed introduction to geometry of the Gromov--Hausdorff distance can be found in~\cite[Ch.~7]{BurBurIva} or in~\cite{ITlectHGH}.

The case of arbitrary metric spaces is also very interesting. For such spaces many authors use some modifications of the Gromov--Hausdorff distance. For example, pointed spaces, i.e., the spaces with marked points, are considered, see~\cite{Jen} and~\cite{Herron}. In the present paper we continue the studying of the classical Gromov--Hausdorff distance on the classes $\GH$ and $\cB$ consisting of the representatives of isometry classes of all metric spaces and of bounded metric spaces, respectively, that we started in~\cite{BIT}. Here the term ``class'' is understood in the sense of von Neumann--Bernays--G\"odel set theory (NBG) that permits to define correctly the distance on the proper class $\GH$, and to construct on it an analogue of metric topology basing on so-called filtration by sets with respect to cardinality, see below. In paper~\cite{BIT}, continuous curves in $\GH$ are defined, and it is proved that the Gromov--Hausdorff distance is an interior generalized semi-metric as on $\GH$, so as on $\cB$, i.e., the distance between points is equal to the infimum of the lengths of curves connecting these points.

In the present paper the questions of path connectivity in $\GH$ are studied, in particular, the path connected components of $\GH$ are described. One of these components is the class $\cB$ of bounded metric spaces. Notice that the geometry of those components is rather tricky, see~\cite{BogaTuz}. Also in the present paper the path connectivity of spheres is proved for several particular cases. Namely, the following results are obtained.
\begin{itemize}
\item It is shown that all spheres centered at the single-point metric space in $\GH$, in $\cB$, and in $\cM$ are path connected.
\item For each bounded metric space $X$, there exists a real $R_X$ such that all the spheres centered at $X$ with radius $r\ge R_X$ are path connected in $\cB$. And if $X$ is a compact space, then such spheres are path connected in $\cM$.
\item For each generic metric space $M$  (see the definition below), there exists a real $r_M>0$ such that all the spheres centered at $M$ with radius $r\le r_M$ are path connected in $\GH$. And if $M$ is a bounded  (respectively, compact) metric space, then such spheres are path connected in $\cB$ (in $\cM$, respectively).
\end{itemize}

\subsubsection*{Acknowledgements}
This work was done at Lomonosov Moscow State University, Interdisciplinary Science Educational School ``Mathematical Methods of Complex Systems Analysis''. The authors use this opportunity to express their deep respect to Professor Yuri M.~Smirnov, whose books, poetry and lectures about classical opera had great influence on us. The work is supported by Russian Science Foundation, Project 21--11--00355.

\section{Preliminaries}
\markright{\thesection.~Preliminaries}
\noindent Let $X$ be an arbitrary set. By $\#X$ we denote the cardinality of $X$, and let $\cP_0(X)$ be the set of all its non-empty subsets.  A symmetric mapping $d\:X\x X\to[0,\infty]$ vanishing on the pairs of the same elements is called a \emph{distance function on $X$}. If $d$ satisfies the triangle inequality, then $d$ is referred as a \emph{generalized semi-metric}. If in addition, $d(x,y)>0$ for all $x\ne y$, then $d$ is called a \emph{generalized metric}. At last, if $d(x,y)<\infty$ for all $x,y\in X$, then such distance function is called a  \emph{metric}, and sometimes a \emph{finite metric\/} to emphasize the difference with a generalized metric. A set $X$ endowed with a  (generalized) (semi-)metric is called a \emph{\(generalized\/\) \(semi-\/\)metric space}.

We need the following simple properties of metrics.

\begin{prop}\label{prop:metric}
The following statements are valid.
\begin{enumerate}
\item \label{prop:metric:2} A non-trivial non-negative linear combination of two metrics given on an arbitrary set is also a metric.
\item \label{prop:metric:3} A positive linear combination of a metric and a semi-metric given on an arbitrary set is a metric.
\end{enumerate}
\end{prop}

If $X$ is a set endowed with a distance function, then the distance between its points $x$ and $y$ we usually denote by $|xy|$. To emphasize that the distance between $x$ and $y$ is calculated in the space $X$, we write $|xy|_X$. Further, if $\g\:[a,b]\to X$ is a continuous curve in $X$, then its \emph{length $|\g|$} is defined as the supremum of the ``lengths of inscribed polygonal lines'', i.e., of the values $\sum_i\big|\g(t_i)\g(t_{i+1})\big|$, where the supremum is taken over all possible partitions $a=t_1<\cdots<t_k=b$ of the segment $[a,b]$. A distance function on $X$ is called \emph{interior\/} if the distance between any two its points $x$ and $y$ equals to the infimum of lengths of curves connecting these points. A curve $\g$ whose length differs from $|xy|$ at most by $\e$ is called \emph{$\e$-shortest}. If for any pair of points $x$ and $y$ of the space $X$ there exists a curve whose length equals to the infimum of lengths of curves connecting these points and equals $|xy|$, then the distance is called \emph{strictly interior}, and the space $X$ is referred as  a \emph{geodesic space}.

Let $X$ be a metric space. For each $A,\,B\in\cP_0(X)$ and each $x\in X$ we put
\begin{flalign*}
\indent&|xA|=|Ax|=\inf\bigl\{|xa|:a\in A\bigr\},\qquad |AB|=\inf\bigl\{|ab|:a\in A,\,b\in B\bigr\},&\\
\indent&d_H(A,B)=\max\{\sup_{a\in A}|aB|,
\,\sup_{b\in B}|Ab|\}=\max\bigl\{\sup_{a\in A}\inf_{b\in B}|ab|,\,\sup_{b\in B}\inf_{a\in A}|ba|\bigr\}.
\end{flalign*}
The function $d_H\:\cP_0(X)\x\cP_0(X)\to[0,\infty]$ is called the \emph{Hausdorff distance}. It is well-known, see for example~\cite{BurBurIva} or~\cite{ITlectHGH}, that $d_H$ is a metric on the family $\cH(X)\ss\cP_0(X)$ of all non-empty closed bounded subsets of $X$.

Let $X$ and $Y$ be metric spaces. A triple $(X',Y',Z)$ consisting of a metric space $Z$ and two its subsets $X'$ and $Y'$ isometric to $X$ and $Y$, respectively, is called a \emph{realization of the pair $(X,Y)$}. The \emph{Gromov--Hausdorff distance $d_{GH}(X,Y)$ between $X$ and $Y$} is defined as the infimum of reals $r$ such that there exists a realization $(X',Y',Z)$ of the pair $(X,Y)$ with $d_H(X',Y')\le r$.

Notice that the Gromov--Hausdorff distance could take as finite, so as infinite values, and always satisfies the triangle inequality, see~\cite{BurBurIva} or~\cite{ITlectHGH}. Besides, this distance always vanishes at each pair of isometric spaces, therefore, due to triangle inequality, the Gromov--Hausdorff distance does not depend on the choice of representatives of isometry classes. There are examples of non-isometric metric spaces with zero Gromov--Hausdorff distance between them, see~\cite{Ghanaat}.

Since each set can be endowed with a metric (for example, one can put all the distances between different points to be equal to $1$), then representatives of isometry classes form a proper class. This class endowed with the Gromov--Hausdorff distance is denoted by $\GH$. Here we use the concept  \emph{class\/} in the sense of von Neumann--Bernays--G\"odel set theory (NBG).

Recall that in NBG all objects (analogues of ordinary sets) are called \emph{classes}. There are two types of classes: \emph{sets\/}  (the classes that are elements of other classes), and \emph{proper classes\/} (all the remaining classes). The class of all sets is an example of a proper class. Many standard operations are well-defined for classes. Among them are intersection, complementation, direct product, mapping, etc.

Such concepts as a \emph{distance function, \(generalized\/\) semi-metric\/} and \emph{\(generalized\/\) metric\/} are defined in the standard way for any class, as for a set, so as for a proper class, because the direct products and mappings are defined. But direct transfer of some other structures, such as topology, leads to contradictions.  For example, if we defined a topology for a proper class, then this class has to be an element of the topology, that is impossible due to the definition of proper classes. In paper~\cite{BIT} the following construction is suggested.

For each class $\cC$  consider a  ``filtration'' by subclasses $\cC_n$, each of which consists of all the elements of $\cC$ of cardinality at most $n$, where $n$ is a cardinal number. Recall that elements of a class are sets, therefore cardinality is defined for them. A class $\cC$ such that all its subclasses $\cC_n$ are sets is said to be  \emph{filtered by sets}. Evidently, if a class $\cC$ is a set, then it is filtered by sets.

Thus, let $\cC$ be a class filtered by sets. When we say that the class $\cC$ satisfies some property, we mean the following: Each set $\cC_n$ satisfies this property. Let us give several examples.
\begin{itemize}
\item Let a distance function on $\cC$ be given. It induces an ``ordinary'' distance function on each set $\cC_n$. Thus, for each $\cC_n$ the standard concepts of metric geometry such as open balls and spheres are defined and are sets. The latter permits to construct standard metric topology $\t_n$ on $\cC_n$ taking the open balls as a base of the topology. It is clear that if $n\le m$, then $\cC_n\ss\cC_m$, and the topology $\t_n$ on $\cC_n$ is induced from $\t_m$.
\item More general, a \emph{topology\/} on the class $\cC$ is defined as a family of topologies $\t_n$ on the sets  $\cC_n$ satisfying the following \emph{consistency condition}: If $n\le m$, then $\t_n$ is the topology on $\cC_n$ induced from $\t_m$. A class endowed with a topology is referred as a \emph{topological class}.
\item The presence of a topology permits to define continuous mappings from a topological space $Z$ to a topological class $\cC$. Notice that according to NBG axioms,  for any mapping $f\:Z\to\cC$ from the set  $Z$ to the class $\cC$, the image $f(Z)$ is a set, all elements of $f(Z)$ are also some sets, and hence, the union  $\cup f(Z)$ is a set of some cardinality $n$. Therefore, each element of $f(Z)$ is of cardinality at most  $n$, and so, $f(Z)\ss\cC_n$. The mapping $f$ is called \emph{continuous}, if $f$ is a continuous mapping from  $Z$ to $\cC_n$. The consistency condition implies that for any $m\ge n$, the mapping $f$ is a continuous mapping from  $Z$ to $\cC_m$, and also for any $k\le n$ such that $f(Z)\ss\cC_k$, the mapping $f$ considered as a mapping from $Z$ to $\cC_k$ is continuous.
\item The above arguments allows to define \emph{continuous curves in a topological class $\cC$}.
\item Let a class $\cC$ be endowed with a distance function and the corresponding topology. We say that the distance function is \emph{interior\/} if it satisfies the triangle inequality, and for any two elements from $\cC$ such that the distance between them is finite, this distance equals the infimum of the lengths of the curves connecting these elements.
\item Let a sequence $\{X_i\}$ of elements from a topological class $\cC$ be given. Since the family $\{X_i\}_{i=1}^\infty$ is the image of the mapping $\N\to\cC$, $i\mapsto X_i$, and $\N$ is a set, then, due to the above arguments, all the family $\{X_i\}$ lies in some $\cC_m$. Thus, the concept of  \emph{convergence\/} of a sequence in a topological class is defined, namely, the sequence converges if it converges with respect to some topology $\t_m$ such that $\{X_i\}\ss\cC_m$, and hence, with respect to any such topology.
\end{itemize}

Our main examples of topological classes are the classes $\GH$ and $\cB$ defined above. Recall that the class $\GH$ consists of representatives of isometry classes of all metric spaces, and the class $\cB$ consists of representatives of isometry classes of all bounded metric spaces. Notice that $\GH_n$ and $\cB_n$ are sets for any cardinal number $n$.

The most studied subset of $\GH$ is the set of all compact metric spaces. It is called a \emph{Gromov--Hausdorff space\/} and often denoted by $\cM$. It is well-known, see~\cite{BurBurIva, ITlectHGH, INT}, that  the Gromov--Hausdorff distance is an interior metric on $\cM$, and the metric space $\cM$ is Polish and geodesic. In paper~\cite{BIT} it is shown that the Gromov--Hausdorff distance is interior as on the class $\GH$, so as on the class $\cB$.

As a rule, to calculate the Gromov--Hausdorff distance between a pair of given metric spaces is rather difficult, and for today the distance is known for a few pairs of spaces, see for example~\cite{GrigIT_Sympl}. The most effective approach for this calculations is based on the next equivalent definition of the Gromov--Hausdorff distance, see details in~\cite{BurBurIva} or~\cite{ITlectHGH}.

Recall that a \emph{relation\/} between sets $X$ and $Y$ is defined as an arbitrary subset of their direct product $X\x Y$. Thus, $\cP_0(X\x Y)$ is the set of all non-empty relations between $X$ and $Y$.

\begin{dfn}
For any $X,Y\in\GH$ and any $\s\in\cP_0(X\x Y)$, the \emph{distortion $\dis\s$ of the relation $\s$} is defined as the following value:
$$
\dis\s=\sup\Bigl\{\bigl||xx'|-|yy'|\bigr|:(x,y),\,(x',y')\in\s\Bigr\}.
$$
\end{dfn}

A relation $R\ss X\x Y$ between sets $X$ and $Y$ is called a \emph{correspondence}, if the restrictions of the canonical projections $\pi_X\:(x,y)\mapsto x$ and $\pi_Y\:(x,y)\mapsto y$ onto $R$ are surjective. Notice that relations can be considered as partially defined multivalued mapping. Form this point of view, the correspondences are multivalued surjective mappings. For a correspondence $R\ss X\x Y$ and $x\in X$, we put $R(x)=\big\{y\in Y:(x,y)\in R\big\}$ and call $R(x)$ the \emph{image of the element $x$ under the relation $R$}. By $\cR(X,Y)$ we denote the set of all correspondences between $X$ and $Y$. The following result is well-known.

\begin{ass}\label{ass:GH-metri-and-relations}
For any $X,Y\in\GH$, it holds
$$
d_{GH}(X,Y)=\frac12\inf\bigl\{\dis R:R\in\cR(X,Y)\bigr\}.
$$
\end{ass}

We need the following estimates that can be easily proved by means of Assertion~\ref{ass:GH-metri-and-relations}. By $\D_1$ we denote the single-point metric space.

\begin{ass}\label{ass:estim}
For any $X,Y\in\GH$, the following relations are valid\/\rom:
\begin{itemize}
\item $2d_{GH}(\D_1,X)=\diam X$\rom;
\item $2d_{GH}(X,Y)\le\max\{\diam X,\diam Y\}$\rom;
\item If at least one of $X$ and $Y$ is bounded, then $\bigl|\diam X-\diam Y\bigr|\le2d_{GH}(X,Y)$.
\end{itemize}
\end{ass}

For topological spaces $X$ and $Y$, their direct product $X\x Y$ is considered as the topological space endowed with the standard topology of the direct product. Therefore, it makes sense to speak about \emph{closed relations\/} and \emph{closed correspondences}.

A correspondence $R\in\cR(X,Y)$ is called \emph{optimal\/} if $2d_{GH}(X,Y)=\dis R$. The set of all optimal correspondences between $X$ and $Y$ is denoted by $\cR_\opt(X,Y)$.

\begin{ass}[\cite{IvaIliadisTuz, Memoli}]\label{ass:optimal-correspondence-exists}
For any $X,\,Y\in\cM$, there exists as a  closed optimal correspondence, so as a realization $(X',Y',Z)$ of the pair $(X,Y)$ which the Gromov--Hausdorff distance between $X$ and $Y$ is attained at.
\end{ass}

\begin{ass}[\cite{IvaIliadisTuz, Memoli}]
For any $X,\,Y\in\cM$ and each closed optimal correspondence $R\in\cR(X,Y)$, the family $R_t$, $t\in[0,1]$, of compact metric spaces, where $R_0=X$, $R_1=Y$, and for $t\in(0,1)$ the space $R_t$ is the set $R$ endowed with the metric
$$
\bigl|(x,y),(x',y')\bigr|_t=(1-t)|xx'|+t\,|yy'|,
$$
is a shortest curve in $\cM$ connecting $X$ and $Y$, and the length of this curve equals $d_{GH}(X,Y)$.
\end{ass}

Let $X$ be a metric space and $\l>0$  a real number. By $\l\,X$ we denote the metric space obtained from $X$ by multiplication of all the distances by $\l$, i.e., $|xy|_{\l X}=\l|xy|_X$ for any $x,\,y\in X$. If the space $X$ is bounded, then for $\l=0$ we put $\l X=\D_1$.

\begin{ass} \label{ass:l1}
For any $X,Y\in\GH$ and any $\l>0$, the equality $d_{GH}(\l X,\l Y)=\l\,d_{GH}(X,Y)$ holds. If $X,Y\in\cB$, then, in addition, the same equality is valid for $\l=0$.
\end{ass}

\begin{ass}\label{ass:l1l2}
Let $X\in\cB$ and $\l_1,\l_2\ge0$. Then $2d_{GH}(\l_1 X,\l_2 X)=|\l_1-\l_2|\,\diam X$.
\end{ass}

\begin{rk}
If $\diam X=\infty$, then Assertion~\ref{ass:l1l2} is not true in general. For example, let $X=\R$. The space  $\l\,\R$ is isometric to $\R$ for any $\l>0$, therefore $d_{GH}(\l_1\R, \l_2\R)=0$ for any $\l_1,\l_2>0$.
\end{rk}

\section{Path Connectivity.}
\markright{\thesection.~Path Connectivity.}
In the present section, the path connected components of $\cM$, $\cB$, and $\GH$ are described, and the path connectivity of spheres in these classes is investigated.

\subsection{Path Connectivity Components.}
Let us start with several auxiliary results.

\begin{lem}\label{lem:diam}
If $\g\:[a,b]\to\cB$ is a continuous curve, then $\diam \g(t)$ is a continuous function.
\end{lem}

\begin{proof}
According to Assertion~\ref{ass:estim}, we have: $\diam\g(t)=2d_{GH}\bigl(\g(t),\D_1\bigr)$. It remains to use the fact that the distance function is continuous.
\end{proof}

\begin{lem}\label{lem:dist_inf}
If $\g\:[a,b]\to\GH$ is a continuous curve, then the distance between its ends is finite. In particular, if the distance between $X$ and $Y$ is infinite, then such $X$ and  $Y$ can not be connected by a continuous curve.
\end{lem}

\begin{proof}
Indeed, the image of the segment $[a,b]$ under the continuous mapping $\g$ is a compact subset. Therefore, it can be covered by a finite number of balls of a fixed radius. The distance between any two points from such ball is finite. Since the Gromov--Hausdorff distance is interior, then such points can be connected by a curve of finite length. The latter implies that the endpoints of the curve $\g$ can be connected by a  ``polygonal line'' of a finite length.
\end{proof}

\begin{lem}\label{lem:fin_inf}
If $\g\:[a,b]\to\GH$ is a continuous curve, then either $\diam \g(t)=\infty$ for all $t$, or $\diam
\g(t)<\infty$ for all $t$, i.e., in the latter case, the $\g$ is a continuous curve in $\cB$.
\end{lem}

\begin{proof}
It is clear that if $X,\,Y\in\GH$, $\diam X<\infty$ and $\diam Y=\infty$, then $d_{GH}(X,Y)=\infty$, but the distance between any two  points of a continuous curve is finite according to Lemma~\ref{lem:dist_inf}, that implies the lemma's statement.
\end{proof}

\begin{cor}\label{cor:path_conn_GH}
Two spaces $X,\,Y\in\GH$ can be connected by a continuous curve in $\GH$, if and only if $d_{GH}(X,Y)<\infty$. In particular, $\GH$ is not path connected, and its path connected components are the classes of spaces with pairwise finite mutual distances. One of such components coincides with the class $\cB$.
\end{cor}

\begin{proof}
Indeed, if the distance between a pair of points in $\GH$ is infinite, then, due to Lemma~\ref{lem:dist_inf}, there is no a continuous curve connecting them, and if the distance is finite, then such continuous curve exists because the distance is interior. Finally, the distance between any two spaces from $\cB$ is finite according to Assertion~\ref{ass:estim}.
\end{proof}

\begin{lem}\label{lem:scalar}
If  $\l(t)$, $t\in[a,b]$, is a continuous non-negative function and $\g\:[a,b]\to\cB$ a continuous curve, then  $\l(t)\g(t)$ is a continuous curve in $\cB$.
\end{lem}

\begin{proof}
Consider the mapping $F\:[0,\infty)\x\cB\to\cB$ defined as $F\:(\l,X)\mapsto\l\,X$. We will show that $F$ is continuous. Notice that
$$
d_{GH}(\l X,\mu Y)\le d_{GH}(\l X,\mu X)+d_{GH}(\mu X,\mu Y)=\frac12|\l-\mu|\,\diam X+\mu\,d_{GH}(X,Y).
$$
Further, for any $\e>0$, there exists $\dl>0$ such that for any pair $(\mu,Y)$ with $d_{GH}(X,Y)<\dl$ and $|\l-\mu|<\dl$, the inequalities $\frac12|\l-\mu|\,\diam X<\e/2$ and $\mu\,d_{GH}(X,Y)<\e/2$ are valid, and hence, $d_{GH}(\l X,\mu Y)<\e$, and the latter means continuity of $F$. It remains to notice that $\l(t)\g(t)=F\bigl(\l(t),\g(t)\bigr)$.
\end{proof}

\subsection{Path Connectivity of Spheres Centered at Single-Point Metric Space}

For a metric space $X\in\GH$ and a real $r\ge0$, we define the \emph{sphere $S_r(X)$} and the \emph{ball $B_r(X)$ in $\GH$} in the standard way:
$$
S_r(X)=\{Y:d_{GH}(X,Y)=r\}\ \ \text{and}\ \ B_r(X)=\{Y:d_{GH}(X,Y)\le r\}.
$$
As usual, $X$ and $r$ are called the \emph{center\/} and the \emph{radius\/} of the sphere $S_r(X)$ and of the ball $B_r(X)$, respectively. Notice that for $r>0$, the spheres and the balls are proper subclasses in $\GH$. Also notice that for $X\in \cB$, the spheres and the balls in the class $\cB$ (in the space $\cM$ for $X\in\cM$) can be defined as the intersections of $S_r(X)$ and $B_r(X)$ with $\cB$ (with $\cM$, respectively).

\begin{lem}\label{lem:geo}
Let $A$ and $B$ be arbitrary spaces in $\cB$ with nonzero diameters. Then for any positive $\e$ such that $2\e<\min\{\diam A,\diam B\}$, any $\e$-shortest curve connecting $A$ and $B$ does not pass through $\D_1$.
\end{lem}

\begin{proof}
Assume the contrary, and let $\g$ be an $\e$-shortest curve passing through $\D_1$. Denote by $\g_1$ and $\g_2$ the segments of $\g$ between $A$ and $\D_1$ and between $\D_1$ and $B$, respectively. Then $|\g|=|\g_1|+|\g_2|\ge d_{GH}(A,\D_1)+d_{GH}(\D_1,B)=\frac12(\diam A+\diam B)$. On the other hand, $|\g|\le d_{GH}(A,B)+\e\le\frac12\max\{\diam A,\diam B\}+\e$, and hence,  $\diam A+\diam B\le\max\{\diam A,\diam B\}+2\e$. The latter inequality is not valid for $2\e<\min\{\diam A,\diam B\}$, a contradiction.
\end{proof}

\begin{thm}\label{thm:1}
Each sphere centered at the single-point metric space $\D_1$ is path connected.
\end{thm}

\begin{proof}
Let $S=S_r(\D_1)$ be the sphere of radius $r>0$ centered at the single-point space $\D_1$. Then $S\ss\cB$. Consider arbitrary $A,\,B \in S$ and fix an $\e$, $0<2\e<\min\{\diam A,\diam B\}$. Consider an $\e$-shortest curve $\g(t)$ connecting $A$ and $B$. According to Lemmas~\ref{lem:fin_inf} and~\ref{lem:geo}, the diameter of each space $\g(t)$ is finite and does not equal zero. Define a mapping $\dl\:[a,b]\to\mathcal{\cB}$ as follows: $\dl\:t\mapsto\bigl(2r/\diam\g(t)\bigr)\g(t)$. Due to Lemma~\ref{lem:diam}, the function $2r/\diam\g(t)$ is continuous, and so, due to Lemma~\ref{lem:scalar}, we conclude that $\dl(t)$ is a continuous curve. And $\diam \dl(t)=2r$, therefore the curve $\dl$ lies in the sphere $S$. Theorem is proved.
\end{proof}

If $A,\,B\in S_r(\D_1)$ are compact spaces, then one can change the $\e$-shortest curve in the proof of Theorem~\ref{thm:1} by a shortest curve in $\cM$. This implies the following result.

\begin{cor}
Any sphere $S_r(\D_1)$ in $\cM$ centered at the single-point compact space $\D_1$ is linear connected.
\end{cor}

\subsection{Path Connectivity of Large Spheres Centered at an Arbitrary Bounded Metric Space}
We start with several technical results. The next lemma is a special kind of Implicit Function Theorem.

\begin{lem}\label{lem:neyavnaya}
Let $F\:[T_0,T_1]\x[S_0,S_1]\to\R$ be a continuous function that satisfies the following conditions\/\rom:
\begin{enumerate}
\item For each $s\in[S_0,S_1]$, the function  $f_s(t)=F(t,s)$ is strictly monotonic\/\rom;
\item There exists an $r$ such that for all $s\in[S_0,S_1]$ the inequality $F[T_0,s]\le r\le F[T_1,s]$ holds.
Then the set $\big\{(t,s):F(t,s)=r\big\}$ is an image of embedded continuous curve of the form $\g(s) = \big(t(s),s\big)$.
\end{enumerate}
\end{lem}

\begin{proof}
Since the function $f_s(t)$ is strictly monotonic and continuous, and $f_s(T_0)\le r\le f_s(T_1)$, then for each $s$ there exists unique $t=t(s)$ such that $f_s\big(t(s)\big)=r$. Put $\g(s)=\big(t(s),s\big)$. Notice that  $F\big(\g(s)\big)=r$ by definition of $\g$. It remains to show that function $t(s)$ is continuous.

Assume to the contrary that $t(s)$ is discontinuous at some point $s_0$, and let $t_0=t(s_0)$. Then there exists an $\e>0$ and a sequence $s_i\to s_0$ as $i\to\infty$ such that $\bigl|t_0-t(s_i)\bigr|\ge\e$. Due to compactness arguments, the sequence $t_i=t(s_i)$ contains a convergent subsequence. Assume without loss of generality that the entire sequence $t_i$ converges to some $t'$. Since $|t_0-t_i|\ge\e$ for all $i$, then $t'\ne t_0$. Since  $F$ is continuous and $F(s_i,t_i)=r$, then $F(s_0,t')=r$ too. But $F(s_0,t_0)$ is also equal to $r$, that is impossible, because $f_{s_0}$ is strictly monotonic, a contradiction.
\end{proof}

\begin{lem}\label{lem:mon}
For a pair $C,G\in\cB$ of bounded metric spaces with $2d_{GH}(G,C)>\diam G>0$, the function $h(\l):=d_{GH}(G,\l C)$ is strictly monotonically increasing for $\l\ge1$.
\end{lem}

\begin{proof}
Due to assumptions, for any correspondence $R\in\cR(G,C)$ we have $\dis R>\diam G$. Consider a sequence $\bigl((g_i,c_i),(g'_i,c'_i)\bigr)\in R\x R$ such that $\bigl||c_ic'_i|-|g_ig'_i|\bigr|\to\dis R$. Since $\dis R>\diam G$, then the inequality $|c_ic'_i|>|g_ig'_i|$ is valid for sufficiently large $i$. Assume without loss of generality that it holds for all $i$.

By $R_\l$ we denote the correspondence from $\cR(G,\l C)$ that coincides with $R$ as a subset of $G\x C$. Then for $\l\ge1$ we have
$$
\dis R_\l\ge\l|c_ic'_i|-|g_ig'_i|=(\l-1)|c_ic'_i|+|c_ic'_i|-|g_ig'_i|,
$$
and therefore, passing to the limit as $i\to\infty$, we conclude that
$$
\dis R_\l\ge(\l-1)\liminf_{i\to\infty}|c_ic'_i|+\dis R.
$$
Since $\dis R>\diam G>0$, then $|c_ic'_i|>|g_ig'_i|+\frac12\diam G$ for large $i$, therefore, $\liminf_{i\to\infty}|c_ic'_i|\ge\frac12\,\diam G$, and hence, $\dis R_\l\ge\dis R+\frac{\l-1}2\,\diam G$. Since the latter inequality is valid for all correspondences $R\in G\x C$, we conclude that $d_{GH}(G,\l C)\ge d_{GH}(G,C)+\frac{\l-1}4\,\diam G$, and so $d_{GH}(G,\l C)>d_{GH}(G,C)$ for $\l>1$. In particular, each $\l C$, $\l\ge1$, taking instead of $C$, satisfies the conditions of the lemma. So, choosing arbitrary $\l_1,\,\l_2$, $1\le\l_1<\l_2$, and substituting to the inequality $\l_1 C$ and  $\l_2/\l_1$ instead of $C$ and $\l$, respectively, we get $d_{GH}(G,\l_2C)>d_{GH}(G,\l_1C)$.
\end{proof}

Let us now prove the path connectivity of spheres of a sufficiently large radius centered at an arbitrary $G\in\cB$.

\begin{thm}\label{thm:bounded}
For any bounded metric space $G\in\cB$ and any $r>\diam G$, the sphere $S_r(G)$ is path connected.
\end{thm}

\begin{proof}
Since $\diam G<\infty$, then $S_r(G)\ss\cB$, see Assertion~\ref{ass:estim}. Since the case $\diam G=0$ has been considered in Theorem~\ref{thm:1}, assume that $\diam G>0$.

\begin{lem}\label{lem:out}
The ball $B_r(G)$ lies inside the ball $B_{3r/2}(\D_1)$.
\end{lem}

\begin{proof}
Indeed, if $X\in B_r(G)$, then according to the triangle inequality, we have
$$
d_{GH}(\D_1,X)\le d_{GH}(\D_1,G)+d_{GH}(G,X)\le\frac12\,\diam G+r<\frac{3r}2,
$$
and hence, $X\in B_{3r/2}(\D_1)$.
\end{proof}

\begin{lem}\label{lem:in}
The sphere $S_r(G)$ lies outside the ball $B_{r/2}(\D_1)$.
\end{lem}

\begin{proof}
Indeed, if $X\in S_r(G)$, then, due to the triangle inequality,
$$
d_{GH}(\D_1,X)\ge d_{GH}(X,G)-d_{GH}(G,\D_1)>r-r/2=r/2,
$$
and so, $d_{GH}(\D_1,X)>r/2$.
\end{proof}

Consider a pair of arbitrary points $A$ and $B$ lying in the sphere $S_r(G)$. Due to Lemma~\ref{lem:in}, $\diam A>r$ and $\diam B>r$, therefore, there exists $\l_A>0$ and $\l_B>0$ such that $\diam \l_A A=r$ and $\diam \l_B B=r$, i.e., the points $A'=\l_A A$ and $B'=\l_B B$ lie in the sphere $S_{r/2}(\D_1)$. Due to Theorem~\ref{thm:1}, there exists a continuous curve $\g\:[a,b]\to\cB$ lying in $S_{r/2}(\D_1)$ and such that $A'=\g(a)$ and $B'=\g(b)$. But then, according to Lemma~\ref{lem:scalar}, the points $A''=3A'$ and $B''=3B'$ can be connected by a continuous curve $3\g(s)$ lying in $S_{3r/2}(\D_1)$.

Since
$$
2d_{GH}\bigl(G,\g(s)\bigr)\ge\diam\g(s)-\diam G>r>\diam G,
$$
then Lemma~\ref{lem:mon} can be applied to the points $\g(s)$, and  according to this Lemma, the function  $f_s(t)=d_{GH}\bigl(G,t\,\g(s)\bigr)$, $t\in[1,3]$, is strictly monotonically increasing. Due to Lemmas~\ref{lem:out} and~\ref{lem:in}, for each $s$ there exists $t$ such that $f_s(t)=r$. Let us put $F(t,s)=d_{GH}\bigl(G,t\,\g(s)\bigr)$. Applying Lemma~\ref{lem:neyavnaya}, we obtain a continuous curve lying in $S_r(G)$ and connecting $A$ and $B$. Theorem is proved.
\end{proof}

If $G$ is a compact space, then one can consider a sphere $S_r(G)$ in the Gromov--Hausdorff space $\cM$. It is easy to see that in this case all the curves constructed in the proof of Theorem~\ref{thm:bounded} lie in $\cM$.

\begin{cor}
For each compact metric space $G\in\cM$ and each $r>\diam G$, the sphere $S_r(G)$ in $\cM$ is path connected.
\end{cor}

\subsection{Generic Metric Spaces}
Let $X$ be a metric space, $\#X\ge3$. By $S(X)$ we denote the set of all bijective mappings of the set $X$ onto itself, and let $\id\in S(X)$ be the identical bijection. Put
\begin{align*}
s(X)&=\inf\bigl\{|xx'|:x\ne x',\ x,x'\in X\bigr\},\\
t(X)&=\inf\bigl\{|xx'|+|x'x''|-|xx''|:x\ne x'\ne x''\ne x\bigr\},\\
e(X)&=\inf\bigl\{\dis f:f\in S(X),\ f\ne\id\bigr\}.
\end{align*}

We call a metric space $M\in\GH$, $\#M\ge3$, \emph{generic\/} if all the three its characteristics $s(M)$, $t(M)$, and $e(M)$ are positive.

The following generalization of results from~\cite{ITFiniteLoc} and~\cite{Filin} was obtained by A.~Filin in his diploma.

\begin{ass}\label{ass:diz}
Let $M,X\in\GH$ be metric spaces such that $\#M\ge3$ and $r:=d_{GH}(M,X)<s(M)/2$. Then for any real $d$ such that $r<d\le s(M)/2$, the following statements are valid.
\begin{enumerate}
\item\label{ass:diz:1} There exists a correspondence $R\in\cR(M,X)$ such that $\dis R<2d\le s(M)$.
\item\label{ass:diz:2} For each such $R$, the family $D_R=\bigl\{X_i:=R(i)\bigr\}_{i\in M}$ is a partition of the space $X$.
\item\label{ass:diz:3} For any $i,j\in M$, probably coinciding, and for any $x_i\in X_i$ and $x_j\in X_j$, the estimate $\bigl||x_ix_j|-|ij|\bigr|<2d\le s(M)$ holds.
\item\label{ass:diz:4} For all $i\in M$, the inequality $\diam X_i<2d\le s(M)$ holds.
\item\label{ass:diz:5} If $d\le s(M)/4$, then the partition $D_R$ is uniquely defined, i.e., if $R'\in\cR(M,X)$ is such that $\dis R'<2d$, then $D_{R'}=D_R$.
\item\label{ass:diz:6} If $d\le\min\bigl\{s(M)/4,\,e(M)/4\bigr\}$, then the correspondence $R\in\cR(M,X)$ such that  $\dis R<2d$ is uniquely defined, and hence, $R$ is an optimal correspondence. In this case, $\dis R=2d_{GH}(M,X)$, $\diam X_i\le\dis R$, and $\bigl||x_ix'_j|-|ij|\bigr|\le\dis R$ for all $i,j\in M$, $i\ne j$.
\end{enumerate}
\end{ass}

\begin{proof}
(\ref{ass:diz:1}) Since
$$
d_{GH}(M,X)=\frac12\inf\bigl\{\dis R:R\in\cR(M,X)\bigr\}<d,
$$
then there exists a correspondence $R\in\cR(M,X)$ such that $\dis R<2d\le s(M)$.

(\ref{ass:diz:2}) If for some $x\in X$ one has $\#R^{-1}(x)>1$, then $\dis R\ge\diam R^{-1}(x)\ge s(M)$, a contradiction. Therefore, the family $\big\{X_i:=R(i)\big\}_{i\in M}$ is a partition of the space $X$.

(\ref{ass:diz:3}) Since $\dis R<2d$, then for any $i,j\in M$ and any $x_i\in R(i)$, $x_j\in R(j)$ we have:
$$
\bigl||x_ix_j|-|ij|\bigr|\le\dis R<2d\le s(M).
$$

(\ref{ass:diz:4}) Since $\dis R<2d$, then for any $i\in M$ we have $\diam X_i\le\dis R<2d\le s(M)$.

(\ref{ass:diz:5}) If $d\le s(M)/4$, and $R'\in\cR(M,X)$ is another correspondence such that $\dis R'<2d$, then $\big\{X'_i:=R'(i)\big\}_{i\in M}$ is also a partition of $X$ satisfying the properties already proved above. If $D_R\ne D_{R'}$, then either there exists an $X'_i$ intersecting simultaneously some different $X_j$ and $X_k$, or there exists an  $X_i$ intersecting simultaneously some different $X'_j$ and $X'_k$. Indeed, if there is no such an $X'_i$, then each $X'_i$ is contained in some $X_j$, and hence, $\{X'_p\}$ is a sub-partition of $\{X_q\}$. Since these partitions are different, then some element $X_i$ contains some different  $X'_j$ and $X'_k$.

Thus, assume without loss of generality that $X'_i$ intersects some different $X_j$ and $X_k$ simultaneously. Choose arbitrary $x_j\in X_j\cap X'_i$ and $x_k\in X_k\cap X'_i$. Then
$$
2d>\diam X'_i\ge|x_jx_k|>|jk|-2d\ge s(M)-2d,
$$
and hence, $d>s(M)/4$, a contradiction.

(\ref{ass:diz:6}) Let $R'\in\cR(M,X)$ satisfy $\dis R'<2d$ as in previous Item. Then $D_R=D_{R'}$, as it is already proved above. It remains to verify that $X_i=X'_i$ for each $i\in M$. Assume to the contrary that there exists a non-trivial bijection $f\:M\to M$ such that $X_i=X'_{f(i)}$. Due to assumptions, $\dis f\ge e(M)\ge4d$. The latter means that there exist $i,j\in M$ such that
$$
\Bigl||ij|-\bigl|f(i)f(j)\bigr|\Bigr|\ge4d.
$$
Let us put $p=f(i)$, $q=f(j)$, and choose arbitrary $x_i\in X_i=X'_p$ and $x_j\in X_j=X'_q$. Then
$$
4d\le\bigl||pq|-|ij|\bigr|=\bigl||pq|-|x_ix_j|+|x_ix_j|-|ij|\bigr|\le\bigl||pq|-|x_ix_j|\bigr|+\bigl||x_ix_j|-|ij|\bigr|<2d+2d,
$$
a contradiction. Thus, a correspondence $R$ whose distortion is close enough to $\frac12d_{GH}(M,X)$ is uniquely defined. So, $R$ is an optimal correspondence.
\end{proof}

\begin{lem}\label{lem:M_sigma}
Let $M$ be a metric space, and $s(M)>0$. For each pair $\{i,j\}\ss M$, fix a real number $a_{ij}=a_{ji}$ in such a way that $a_{ii}=0$ and $|a_{ij}|<s(M)$. Change the distance function on $M$ by adding the value $a_{ij}$ to each distance $|ij|$ and denote by $\r$  the resulting function. Put $a=\sup|a_{ij}|$. If $a\le t(M)/3$, then the following statements are valid.
\begin{enumerate}
\item\label{lem:M_sigma:1} The distance function $\r$ is a metric on the set $M$.
\item\label{lem:M_sigma:2} For the resulting metric space $M_\r=(M,\r)$, the inequality $d_{GH}(M,M_\r)\le a/2$ holds.
\item\label{lem:M_sigma:3} If in addition $a/2<\min\bigl\{s(M)/4,\,e(M)/4\bigr\}$, then $d_{GH}(M,M_\r)=a/2$.
\item\label{lem:M_sigma:4} Under assumptions of Item~$(\ref{lem:M_sigma:3})$, if the equality  $a_{ij}=\pm a$ holds for all $i\ne j$, then $M_\r$ lies in the sphere $S_{a/2}(M)$\rom; moreover, if $M_{\r'}$ is constructed in a similar way by the set $a'_{ij}=\pm a$, $i\ne j$, and $\r_t=(1-t)\r+t\,\r'$, $t\in[0,1]$, then the set $M$ with the distance function $\r_t$ is a metric space, the mapping $t\mapsto M^t:=(M,\r_t)$ is a continuous curve in $\GH$, and, if for some pair $\{i,j\}$, $i\ne j$, the equality $a_{ij}=a'_{ij}$ holds, then all the spaces $M^t$ lie in the sphere  $S_{a/2}(M)$.
\end{enumerate}
\end{lem}

\begin{proof}
(\ref{lem:M_sigma:1}) It is sufficient to verify that the distance function $\r$ satisfies the triangle inequalities. Choose arbitrary three points $i,j,k\in M$, then
$$
\r(i,j)+\r(j,k)-\r(i,k)=|ij|+a_{ij}+|jk|+a_{jk}-|ik|-a_{ik}\ge t(M)-3a\ge0.
$$

(\ref{lem:M_sigma:2}) Let $R\in\cR(M,M_\r)$ be the identical mapping, then
$$
2d_{GH}(M,M_\r)\le\dis R=\sup\Bigl\{\bigl||ij|-|ij|-a_{ij}\bigr|:i,j\in M\Bigr\}=\sup\bigl\{|a_{ij}|:i,j\in M\bigr\}=a.
$$

(\ref{lem:M_sigma:3}) The assumptions imply that Assertion~\ref{ass:diz}, Item~(\ref{ass:diz:6}), can be applied, and hence, the identical mapping $R$ is an optimal correspondence. The latter implies the equality required.

(\ref{lem:M_sigma:4}) Proposition~\ref{prop:metric} implies that $M^t$ is a metric space. Further, let  $R^t\in\cR(M,M^t)$ be the identical mapping, then $\dis R^t\le a$, and so $R^t$ is an optimal correspondence. Let $R'\in\cR(M^t,M^s)$ be the identical mapping too. Then
$$
2d_{GH}(M^t,M^s)\le\dis R'=|t-s|\,\sup_{i,j\in M}|a_{ij}-a'_{ij}|\le 2a|t-s|,
$$
therefore, $M^t$ is a continuous curve in $\GH$. It remains to notice that if $a_{ij}=a'_{ij}$ for some $i\ne j$, then
$$
\bigl|\r_t(i,j)-|ij|\bigr|=\bigl|(1-t)a_{ij}+t\,a'_{ij}\bigr|=a,
$$
so $\dis R^t=a$ for all $t$, and hence, $d_{GH}(M,M^t)=a/2$.
\end{proof}

\subsection{Path Connectivity of Small Spheres Centered at Generic Spaces}

We need the following notations. Let $A$ and $B$ be non-empty subsets of a metric space $X$. We put
$$
|AB|=\inf\bigl\{|ab|:a\in A,\,b\in B\bigr\},\qquad  |AB|'=\sup\bigl\{|ab|:a\in A,\,b\in B\bigr\}.
$$
If we need to emphasis that these values are calculated with respect to a metric $\r$, then we write $|AB|_\r$ and $|AB|'_\r$ instead of $|AB|$ and $|AB|'$, respectively. Similarly, we sometimes write $\diam_\r A$ instead of $\diam A$ by the same reason.

\begin{thm}\label{thm:small}
Let $M$ be a generic metric space and
$$
0<r<\min\bigl\{s(M)/4,\,e(M)/4,\,t(M)/6\bigr\}.
$$
Then the sphere $S_r(M)$ in $\GH$ is path connected.
\end{thm}

\begin{proof}
Choose an arbitrary $X\in S_r(M)$. Since the conditions of Assertion~\ref{ass:diz}, Item~(\ref{ass:diz:6}), are valid, then there exists a unique optimal correspondence $R\in\cR(M,X)$. Moreover, $\dis R=2r$ and, besides, the correspondence $R$ defines the partition $D_R=\big\{X_i:=R(i)\big\}_{i\in M}$ of the space $X$ into subsets  $X_i$, whose diameters are at mots $2r$ and such that for any $x_i\in X_i$ and $x_j\in X_j$, the inequality  $\bigl||x_ix_j|-|ij|\bigr|\le2r$ holds.

It is easy to see that the distortion of the correspondence $R$ can be calculated by the following formula:
$$
\dis R=\max\Bigl[\sup_i\{\diam X_i\},\,\sup_{i\ne j}\bigl\{|X_iX_j|'-|ij|\bigr\},\,\sup_{i\ne j}\bigl\{|ij|-|X_iX_j|\bigr\}\Bigr].
$$
We consider three possibilities depending on which of the three items the maximum attains at.

{\bf (1)} Let $\dis R=\sup_i\{\diam X_i\}$. Define a distance function $\r$ on the set $X$ as follows: the distances between the points in each $X_i$ do not change, and for $x_i\in X_i$ and $x_j\in X_j$, $i\ne j$, put $\r(x_i,x_j)=|ij|+2r$. Let us show that $\r$ is a metric. To do that, it suffices to verify the triangle inequalities for the points $x_i,x'_i\in X_i$ and $x_j\in X_j$, $j\ne i$. Since the triangle $x_ix'_ix_j$ is equilateral, it suffices to show that $\r(x_i,x'_i)\le2\r(x_i,x_j)$. We have
$$
2\r(x_i,x_j)-\r(x_i,x'_i)=2|ij|+4r-|x_ix'_i|=\bigl(2|ij|-|x_ix'_i|\bigr)+4r>0.
$$
Notice that $|X_iX_j|_\r=|X_iX_j|'_\r=|ij|+2r$ for any $i\ne j$.

Define the functions $\r_t$, $t\in[0,1]$, that coincide with the initial metric on all $X_i$ and equal $\r_t(x_i,x_j)=(1-t)|x_ix_j|+t\bigl(|ij|+2r\bigr)$ for all $x_i\in X_i$ and $x_j\in X_j$, $i\ne j$. All the functions $\r_t$ are metrics according to Proposition~\ref{prop:metric}. By $X^t$ we denote the metric space $(X,\r_t)$. If  $R'\in\cR(X^t,X^s)$ is the identical mapping, then
$$
\dis R'=|t-s|\sup\Bigl\{\bigl||x_ix_j|-|ij|-2r\bigr|:i,j\in M,\,i\ne j\Bigr\}\le 4r\,|t-s|,
$$
and so $X^t$ is a continuous curve in $\GH$.

Further, let us show that all the spaces $X^t$ lie in the sphere $S_r(M)$. Since the metric $\r_t$ coincides with the initial one at each $X_i$, then $\diam_{\r_t}X_i=\diam X_i$. Further, for $i\ne j$, the function $\r_t(x_i,x_j)=|x_ix_j|+t\bigl(|ij|+2r-|x_ix_j|\bigr)$ increases monotonically as $t$ increases, and
$$
|x_ix_j|\le\r_t(x_i,x_j)\le|ij|+2r.
$$
Therefore,
$$
-2r\le|X_iX_j|'-|ij|\le|X_iX_j|'_{\r_t}-|ij|\le2r\quad \text{and}\quad 2r\ge|ij|-|X_iX_j|\ge|ij|-|X_iX_j|_{\r_t}\ge-2r.
$$
So, we conclude that
$$
\sup_{i\ne j}\bigl\{|X_iX_j|'_{\r_t}-|ij|\bigr\}\le2r\quad \text{and}\quad  \sup_{i\ne j}\bigl\{|ij|-|X_iX_j|_{\r_t}\bigr\}\le2r.
$$
But $\sup_i\{\diam_{\r_t} X_i\}=\sup_i\{\diam X_i\}=2r$, therefore, if we take the correspondence  $R^t\in\cR(M,X^t)$ coinciding with $R$, then
$$
\dis R^t=\max\Bigl[\sup_i\{\diam_{\r_t}X_i\},\,\sup_{i\ne j}\bigl\{|X_iX_j|'_{\r_t}-|ij|\bigr\},\,\sup_{i\ne j}\bigl\{|ij|-|X_iX_j|_{\r_t}\bigr\}\Bigr]=2r.
$$
So, each space $X^t$ satisfies the conditions of Assertion~\ref{ass:diz}, Item~(\ref{ass:diz:6}), and hence, $R^t$ is an optimal correspondence. Thus, $2d_{GH}(M,X^t)=\dis R^t=2r$, i.e., the curve $X^t$ lies in $S_r(M)$.

Now let us deform the metric on the space $X^1$ as follows: for points $x,\,x'$ lying in the same $X_i$, we put $\nu_t(x,x')=(1-t)|xx'|$, $t\in[0,1]$, and for points lying in different sets $X_i$, we do not change the metric. It is easy to see that all $\nu_t$ are metrics, and for the metric space $X^1_t=(X,\nu_t)$, the next relations are valid:
\begin{gather*}
\diam_{\nu_t}X_i=(1-t)\,\diam X_i\le2r,\\
|X_iX_j|'_{\nu_t}-|ij|=|X_iX_j|'_\r-|ij|=2r,\\
|ij|-|X_iX_j|_{\nu_t}=|ij|-|X_iX_j|_\r=-2r.
\end{gather*}
Consider the correspondence $R$ as a correspondence between $M$ and $X^1_t$, and denote it by   $R^t_1\in\cR(M,X^1_t)$. Then $\dis R^t_1=2r$, therefore, $R^t_1$ is an optimal correspondence, and $d_{GH}(M,X^1_t)=r$, i.e., all the spaces $X^1_t$ lie in the sphere $S_r(M)$. The same arguments as above show that the mapping $t\mapsto X^1_t$ is a continuous curve in $\GH$. The space $X^1_1$ is obtained from $M$ by adding the value $2r$ to all the distances $|ij|$, $i\ne j$. This space is denoted by $M^+$.

Thus, we have constructed a continuous curve in $S_r(M)\subset\GH$ connecting the space $X\in S_r(M)$ with the space $M^+\in S_r(M)$.

{\bf (2)} Let $\dis R=\sup_{i\ne j}\bigl\{|X_iX_j|'-|ij|\bigr\}$. Consider the same family $\r_t$ of metrics as in Case~(1). Since
$$
|X_iX_j|'-|ij|\le |X_iX_j|'_{\r_t}-|ij|\le2r,
$$
and $\sup_{i\ne j}\bigl\{|X_iX_j|'-|ij|\bigr\}=2r$, then $\sup_{i\ne j}\bigl\{|X_iX_j|'_{\r_t}-|ij|\bigr\}=2r$ for all $t$. As it is shown in Case~(1), $\diam_{\r_t}X_i=\diam X_i\le 2r$ and $\big||ij|-|X_iX_j|_{\r_t}\big|\le2r$, therefore, the family $X^t=(X,\r_t)$ lies in the sphere $S_r(M)$. The arguments from Case~(1) prove the continuity of the curve  $t\mapsto X^t$, whose ending point can be connected with the space $M^+\in S_r(M)$ by means of the same construction. Thus, again we have constructed a continuous curve in $S_r(M)\subset\GH$ connecting the space $X\in S_r(M)$ with the space $M^+\in S_r(M)$.

{\bf (3)} Let $\dis R=\sup\bigl\{|ij|-|X_iX_j|:i\ne j\bigr\}$. In this case, we define the metric $\r$ in a different way, namely, we put $\r(x_i,x_j)=|ij|-2r$ for $i\ne j$, and $\r_t(x_i,x_j)=(1-t)|x_ix_j|+t\bigl(|ij|-2r\bigr)$. Let us show that $\r$ is a metric too. The value $|ij|-2r$ is positive, because $r<s(M)/4$, and hence $|ij|-2r>|ij|-s(M)/2\ge s(M)/2>0$. It remains to verify the triangle inequalities. If all three points belong to the same $X_i$, then the inequality coincides with the one in $X$, and hence it holds. Further, if $x_i,x'_i\in X_i$ and $x_j\in X_j$, $j\ne i$, then the triangle $x_ix'_ix_j$ is equilateral, and it suffices to show that $\r(x_i,x'_i)\le2\r(x_i,x_j)$. Since $r<s(M)/4$, we have
$$
2\r(x_i,x_j)-\r(x_i,x'_i)=2|ij|-4r-|x_ix'_i|>2|ij|-s(M)-|x_ix'_i|=\big(|ij|-s(M)\big)+\big(|ij|-|x_ix'_i|\big).
$$
It remains to notice that $|ij|-s(M)\ge 0$ by definition of $s(M)$, and $|ij|-|x_ix'_i|>2r$ because $\diam X_i\le 2r$ and $|ij|\ge s(M)>4r$. At last, if $x_i\in X_i$, $x_j\in X_j$, and $x_k\in X_k$ for pairwise distinct $i,\,j,\,k$, then
$$
\r(x_i,x_j)+\r(x_j,x_k)-\r(x_i,x_k)=|ij|+|jk|-|ik|-2r\ge t(M)-2r>6r-2r=4r>0.
$$
Thus, $\r$ is a metric. Notice that $|X_iX_j|_\r=|X_iX_j|'_\r=|ij|-2r$ for any $i\ne j$.

In this case the function $\r_t(x_i,x_j)=|x_ix_j|-t\bigl(2r+|x_ix_j|-|ij|\bigr)$ decreases monotonically as $t$ increases. Since
$$
|ij|-|X_iX_j|\le|ij|-|X_iX_j|'_{\r_t}\le2r,
$$
and $\sup_{i\ne j}\bigl\{|ij|-|X_iX_j|\bigr\}=2r$, then $\sup_{i\ne j}\bigl\{|ij|-|X_iX_j|_{\r_t}\bigr\}=2r$ for all $t$. Reasonings similar to Cases~(1) and~(2) show that $\diam_{\r_t}X_i=\diam X_i\le 2r$ and $\big||X_iX_j|'_{\r_t}-|ij|\big|\le2r$, therefore, the family $X^t=(X,\r_t)$ forms a continuous curve lying in the sphere $S_r(M)$. Further, the arguments  similar to the ones given in the previous cases allow us to construct a continuous curve in $\GH$ connecting the space $X\in S_r(M)$ with the space $M^-\in S_r(M)$ by a curve lying in the sphere $S_r(M)$. Here the space $M^-$ is obtained from the space $M$ by subtracting the value $2r$ from all the distances $|ij|$, $i\ne j$.

To conclude the proof, we apply Lemma~\ref{lem:M_sigma}, Item~(\ref{lem:M_sigma:4}), and construct continuous curves in $S_r$ connecting the spaces $M^+$ and $M^-$ with the space $M^{\pm}$, where some nonzero distances have the form $|ij|+2r$, $i\ne j$, and the remaining ones have the form $|ij|-2r$, $i\ne j$. That can be done since $\#M\ge3$ according to the assumptions. Theorem is proved.
\end{proof}

If the space $M$ is finite, and $X$ is a compact metric space, then it is easy to see that all the curves constructed in the proof of Theorem~\ref{thm:small} lies in $\cM$. Thus, the following result holds.

\begin{cor}
Let $M$ be a finite generic space, and
$$
0<r<\min\bigl\{s(M)/4,\,e(M)/4,\,t(M)/6\bigr\}.
$$
Then the sphere $S_r(M)$ in $\cM$ is path connected.
\end{cor}

In conclusion, we give several examples of generic spaces.

\begin{examp}
A finite metric space $M$ with pairwise distinct non-zero distances obviously has $s(M)>0$ and $e(M)>0$. Adding a positive constant $\e$ to all non-zero distances we get $t(M)\ge\e$ in addition.
\end{examp}

The  following elegant construction has been suggested to us by Konstantin Shramov and gives an opportunity to obtain a generic metric space of an arbitrary cardinality.

\begin{examp}
Consider an arbitrary infinite set $X$. It is well-known (Zermelo's theorem) that any set can be well-ordered, i.e., there exists a total order such that any subset contains a least element. It is easy to show, see for example~\cite[Ch.~2]{StRoman}, that an order-preserving bijection of a well-ordered set is identical.

Let us fix some well-order on $X$ and denote it by $\le$. Construct an oriented graph $G_o$ on the vertex set $X$ connecting by an oriented edge $(x,y)$ each pair of vertices $x,\,y\in X$ such that $x\lvertneqq y$. It follows from the above that the automorphism group of this graph is trivial. Now reconstruct $G_o$ to another oriented graph $H_o$ changing each its oriented edge $e=(x,y)$ by three vertices $u_e,\,v_e,\,w_e$ and four oriented edges $(x,u_e)$, $(u_e,v_e)$, $(v_e,y)$, and $(v_e,w_e)$. Finally, denote by $H$ the non-oriented graph corresponding to $H_o$.

Show that the automorphism group of the graph $H=(V_H,E_H)$ is trivial. Let $\s\:V_H\to V_H$ be an automorphism of the graph  $H$. Notice that the set $X\ss V_H$ is the set of all vertices from $H$ having infinite degree, therefore the restriction of $\s$ onto $X$ is a bijection of $X$ onto itself. Further, let $e=(x,y)$ be an oriented edge of the graph $G_o$. Let us show that $\big(\s(x),\s(y)\big)$ is also an edge of the graph $G_o$. Assume the contrary, then $f=\big(\s(y),\s(x)\big)$ is an edge of $G_o$ because the order is total. The path $x,\,u_e,\,v_e,\,y$ in $H$ is mapped to the path $\s(x),\,\s(u_e),\,\s(v_e),\,\s(y)$. On the other hand, in $H$ there exists unique $3$-edge path corresponding to the edge $f$, namely, the path $\s(y),\,u_f,\,v_f,\,\s(x)$, therefore $u_f=\s(v_e)$ and $v_f=\s(u_e)$. But  $\deg u_e= \deg u_f =2$, and $\deg v_e= \deg v_f =3$, a contradiction, so $\big(\s(x),\s(y)\big)$ is an edge of $G_o$. The latter means that $\s|_X$ is an automorphism of the oriented graph $G_o$, and hence,  $\s|_X=\id_X$. But then the vertices of each path $x,\,u_e,\,v_e,\,y$ are also fixed, and hence,  $\s=\id_{V_H}$.

Notice that the cardinality of $X$ and $V_H$ are the same. Construct a metric on $V_H$ as follows: put the distances between distinct adjacent vertices equal $1$, and the distance between non-adjacent vertices equal $1+\e$, where $\e\in(0,1)$. Then the resulting metric space $V_H$ is generic, because $s(M)=1$, $t(M)=1-\e$, and $e(M)=\e$.
\end{examp}

This example can be modified easily to obtain an unbounded generic metric space.

\begin{examp}
Let $X$ be an arbitrary generic metric space of a finite diameter $d$, in particular $\#X\ge3$. Put $Y=X\sqcup\{y\}$ and extend metric from $X$ to $Y$ as follows: $|xy|=f$, where $f>d$. It is clear that $s(Y)=s(X)$ and $t(Y)$ remains positive, because all new triangles are isosceles with equal sides of the length $f$ and a smaller side of the  length at least $s(X)>0$. Further, each isometry of $Y$ preserves $y$, and hence, preserves $X$, therefore, $Y$ has no non-trivial isometries. Let $\s$ be a bijection of $Y$.  If $\s$ preserves $y$, then $\dis \s=\dis \s|_X\ge e(X)>0$. Otherwise, take a point $x$ such that $\s(x)\ne y$ (such point exists, because $\#X\ge3$), then $\dis\s\ge\big||yx|-\big|\s(y)\s(x)\big|\big|\ge(f-d)>0$. Thus, $e(Y)>0$ anyway, and hence, $Y$ is a generic metric space also.

This construction permits to obtain an unbounded generic metric space starting from a bounded metric space by adding point by point and taking the distance to next new point greater by $1$ than to the previous one. Notice that if the initial space is infinite, then the resulting space has the same cardinality.
\end{examp}

Next example is based on another idea and gives a countable generic metric space of infinite diameter.

\begin{examp}
Consider a geometric progression $\{q^n\}_{n=0}^\infty\ss\R$, where $q>4$, and define a metric space $X=\{x_i\}$ with the distances $|x_ix_j|=|q^i-q^j|+1$. It is clear that  $s(X)>4$ and $t(X)=1$. Further, let $\s\in S(X)$ be a nontrivial bijection. Then there exists $m$ such that  $\s(x_m)=x_n$ and $n\ne m$. Now, take $k$ as large that $k$ and $l$, where $\s(x_k)=x_l$, are greater than $m$ and $n$.  We have:
$$
\dis\s\ge\big||x_mx_k|-|\s(x_m)\s(x_k)|\big|=\big||x_mx_k|-|x_nx_l|\big|=\big||q^m-q^k|-|q^n-q^l|\big|.
$$
If $k\ne l$, then $\dis\s\ge q^{\max\{k,l\}}\big|1-3/q\big|>q$, because $\max\{k,l\}\ge 3$. Otherwise, $\dis\s\ge|q^m-q^n|\ge q^{\max\{m,n\}}\big|1-1/q\big|\ge q-1$, because $\max\{m,n\}\ge 1$. Thus, $e(X)\ge q-1>3$, and so, $X$ is a generic space.
\end{examp}

\vfill\eject

\markright{References}

\end{document}